\documentclass[twoside,11pt]{article}

\usepackage{jmlr2e,yuyuan}
\usepackage{mathdots} 
\usepackage{mathtools}
\mathtoolsset{showonlyrefs}
\usepackage{arydshln} 
\usepackage[colorlinks=false]{hyperref}
\usepackage{braket} 

\def\atanh{\operatorname{arctanh}}
\newtheorem{assum}{Assumption}

\numberwithin{equation}{section}
\numberwithin{assum}{section}
\numberwithin{theorem}{section}

\ShortHeadings{Logistic Regression Worst-Case Dataset}{Ouyang and Squires}
\firstpageno{1}

\begin{document}

\title{Some Worst-Case Datasets of Deterministic First-Order Methods for Solving Binary Logistic Regression}

\author{\name Yuyuan Ouyang \email yuyuano@clemson.edu \\
       \addr School of Mathematical and Statistical Sciences\\
       Clemson University\\
       Clemson, SC 29634, USA
       \AND
       \name Trevor Squires \email tsquire@clemson.edu \\
       \addr School of Mathematical and Statistical Sciences\\
       Clemson University\\
       Clemson, SC 29634, USA
       }

\editor{}

\maketitle

\begin{abstract}
	We present in this paper some worst-case datasets of deterministic first-order methods for solving large-scale binary logistic regression problems. Under the assumption that the number of algorithm iterations is much smaller than the problem dimension, with our worst-case datasets it requires at least $\cO(1/\sqrt{\varepsilon})$ first-order oracle inquiries to compute an $\varepsilon$-approximate solution. From traditional iteration complexity analysis point of view, the binary logistic regression loss functions with our worst-case datasets are new worst-case function instances among the class of smooth convex optimization problems. 
\end{abstract}







\section{Introduction}
\label{sec:intro}

The following notations will be used throughout this paper. We denote natural logarithm by $\log(\cdot)$.
For any positive integer $k$, we use $\vc 0_k$ and $\vc 1_k$ to denote vectors of all zeros and ones, respectively. When the dimension $k$ is evident, we may remove the subscript and simply use $\vc 0$ and $\vc 1$. We use $\vc e_{t,k}$ to denote the $t$-th standard basis vector in $\R^k$: 
$\vc e_{t,k}^\top = (\vc 0_{t-1}^\top, 1, \vc 0_{k-t}^\top)^\top$. For any vector $\vc u$, we use $u^{(i)}$ to denote the $i$-th component of $\vc u$. 
The norm notation $\|\cdot\|$ is used for the Euclidean norm of a vector and the spectral norm of a matrix.

The main research questions of this paper are the following:

\begin{itemize}
	\em
	\item 
	For any deterministic first-order methods, what is the best possible computational performance on solving large-scale binary logistic regression problems?
	\item 
	For any deterministic first-order methods, what is their respective worst-case datasets of large-scale binary logistic regression problems that yield their worst possible computational performance?
\end{itemize}

Note that we will focus on providing an answer to the second question, since it will lead natural to an answer to the first question. We describe the binary logistic regression problems, and provide the definitions of ``deterministic first-order methods'' and ``computational performance'' in the sequel.

In this paper, we use the following description of binary logistic problems. Given any data matrix $A\in\R^{N\times n}$ and response vector $\vc b\in\{-1,1\}^N$, the binary logistic regression problem is a nonlinear optimization problem that minimizes objective function
\begin{align}
	\label{eq:general_problem}
	\min_{\vc x\in\R^n, y\in\R}\Phi_{A,\vc b}(\vc x,y):=h(A\vc x+y\vc 1) - \vc b^\top(A\vc x+y\vc 1),
\end{align}
where for any $\vc u\in\R^k$, $h$ is defined by 
\begin{align}
	\label{eq:general_h}
	\begin{aligned}
	h(\vc u)\equiv h_k(\vc u):=&\sumt[k]2\log\left[2\cosh\left(\frac{u^{(i)}}{2}\right)\right] 
	\\
	& = \sumt[k]2\log\left[\exp\left(\frac{u^{(i)}}{2}\right)+\exp\left(-\frac{u^{(i)}}{2}\right)\right].
	\end{aligned}
\end{align}
Here $\cosh$ is the hyperbolic cosine function. For convenience we remove the subscript $k$ in the definition of $h$ and allow the variable vector of $h$ to be of any dimension. Using $\vc a_i^\top$ to denote the $i$-th row of $A$, from \eqref{eq:general_problem} and \eqref{eq:general_h} we have
\begin{align}
\label{eq:Phi_reform}
\begin{aligned}
	\Phi(\vc x, y) = & \sumt[N]2\log\left[\exp\left(\frac{\vc a_i^\top \vc x+y}{2}\right) + \exp\left(-\frac{\vc a_i^\top \vc x+y}{2}\right)\right]-b^{(i)}\left(\vc a_i^\top \vc x + y\right)
	\\
	= & \sumt[N]2\log\left[\exp\left(\frac{b^{(i)}(\vc a_i^\top \vc x+y)}{2}\right) + \exp\left(-\frac{b^{(i)}(\vc a_i^\top \vc x+y)}{2}\right)\right]-b^{(i)}\left(\vc a_i^\top \vc x + y\right)
	\\
	= & \sumt[N]2\log\left[1 + \exp\left(-{b^{(i)}(\vc a_i^\top \vc x+y)}\right)\right],
\end{aligned}
\end{align}
which is a commonly used form of binary logistic regression problems with parameter vector $\vc x$ and intercept $y$. Here in the second equality we use the facts that $\cosh$ is an even function and $b^{(i)}\in\{-1,1\}$. Note that we can build an analogy between logistic and least squares problems through the formulation  \eqref{eq:general_problem}: if $h(\cdot):=\|\cdot\|^2/2$ we have a least squares problem immediately. In fact, such analogy has been exploited in \cite{bach2010self} in the analysis of statistical properties of logistic regression.

In this paper, we will make an simplification and assume that we know the value of intercept $y^*$ in an optimal solution $(\vc x^*, y^*)$. Problem \eqref{eq:general_problem} then simplifies to a problem of estimating the parameter vector $\vc x$ from
\begin{align}
	l^*_{A,\vc b}:=\min_{\vc x\in\R^n}l_{A,\vc b}(\vc x):=\Phi_{A,\vc b}(\vc x, y^*).
\end{align}
Indeed, in our designed worst-case dataset,we can show that the intercept $y^*=0$. As a consequence, it suffices to solve a logistic model with homogeneous linear predictor:
\begin{align}
	\label{eq:general_l_homogeneous}
	\begin{aligned}
	l_{A,b}^*:=\min_{\vc x\in\R^n}l_{A,b}(\vc x):=& h(A\vc x) - \vc b^\top A\vc x
	\\
	= & \sumt[N]2\log\left[1 + \exp\left(-{b^{(i)}(\vc a_i^\top \vc x)}\right)\right].
	\end{aligned}
\end{align}

The term ``deterministic first-order method'' is defined by the following oracle description: we say that an iterative algorithm $\cM$ for convex optimization $\min_{\vc x\in\R^n} f(\vc x)$ is a deterministic first-order method if it accesses the information of objective function $f$ through a deterministic first-order \emph{oracle} $\cO_f:\R^n\times \R^n$, such that $\cO_f(\vc x)=(f(\vc x), f'(\vc x))$ for any inquiry $\vc x$, where $f'(\vc x)$ is a subgradient of $f$ at $\vc x$. Specifically, $\cM$ can be described by a problem independent initial iterate $\vc x_0$ and a sequence of rules $\{\cI_t\}_{t=0}^\infty$ such that
\begin{align}
\label{eq:xOiter}
\vc x_{t+1} = \cI_t(\cO_f(\vc x_0), \ldots, \cO_f(\vc x_t)),\ \forall t\ge 0.
\end{align}
Without loss of generality, we can assume that $\vc x_0 = \vc 0$. We also assume that the dimension of the parameter vector $\vc x$ is large and we can only afford $t\ll n$ oracle inquiries. 

The computational performance of $\cM$ is evaluated through its solution accuracy $f(\hat {\vc x}) - f^*$ or $\|\hat{\vc x}-\vc x^*\|$, in which $\hat {\vc x}$ is an approximate solution computed by $\cM$. Without loss of generality, we can assume that $\vc x_t$'s are both inquiry points to the oracle $\cO$ and the approximate solution computed by $\cM$.

\subsection{Related Work}
\label{sec:related_work}

There had been many existing deterministic first-order algorithms that can be applied to solve \eqref{eq:general_l_homogeneous}. For example, applying Nesterov's accelerated gradient method \cite{nesterov2004introductory}, it is known that it takes at most $\cO(1)(1/\sqrt{\varepsilon})$ oracle inquiries to compute an approximate solution $\hat {\vc x}$ to \eqref{eq:general_l_homogeneous} such that $l_{A,\vc b}(\vc x) - l_{A,\vc b}^*\le \varepsilon$. Here $\cO(1)$ is a constant independent of $\varepsilon$. Such result is known as the \emph{upper complexity bound}. Upper complexity bounds depict  achievable computational performance on solving an problem. 

Our research question described at the beginning of this section is focusing on the \emph{lower complexity bound} of a problem, namely, the performance limit of deterministic first-order methods. For convex optimization problems $f^*:=\min_{\vc x}f(\vc x)$, the lower complexity bound is concerned with the least number of inquiries to the deterministic first-order oracle in order to compute an $\varepsilon$-approximate solution $\hat {\vc x}$ such that $f(\hat {\vc x}) - f^*\le \varepsilon$. In the following we list the available lower complexity bound results on deterministic first-order methods for convex optimization $f^*:=\min_{\vc x}f(\vc x)$.

\begin{itemize}
	\item When $f$ is convex (possibly nonsmooth), the lower complexity bound is $\cO(1)(1/\varepsilon^2)$   \cite{nemirovski1983problem,nesterov2004introductory,guzman2015lower,woodworth2016tight}.
	\item When $f$ is weakly smooth convex with parameter $\rho$ (see the definition of weakly smooth in \cite{juditsky2014deterministic}), the lower complexity bound is $\cO(1)(1/\varepsilon^{2/(1+3\rho)})$
	\cite{juditsky2014deterministic,guzman2015lower}.
	\item When $f$ is convex nonsmooth with bilinear saddle point structure, the lower complexity bound is $\cO(1)(1/\varepsilon)$ \cite{ouyang2018lower}.
	\item When $f$ is smooth convex, the lower complexity bound is $\cO(1)(1/\sqrt{\varepsilon})$ \cite{nemirovski1992information,nesterov2004introductory,guzman2015lower,drori2017exact,woodworth2016tight,carmon2017lower,carmon2017lowerII,diakonikolas2018lower}.	
	\item When $f$ is strongly convex smooth, the lower complexity bound is $\cO(1)\log(1/\varepsilon)$ \cite{nesterov2004introductory,woodworth2016tight}.
\end{itemize}

Two remarks are in place for the above list of lower complexity bounds. First, all the lower complexity bounds have been demonstrated to match available upper complexity bounds, namely, there exists available deterministic first-order algorithms that achieves the lower complexity bounds. Such algorithms are known as \emph{optimal algorithms}, since the lower complexity bounds provide the verification that their respective theoretical computational performance would not be improvable anymore. Second, we can observe from the above list that the ``smaller'' the problem class is, the better lower complexity bounds could be. For example, the class of smooth convex optimization problems is a subclass of general convex optimization problems, hence it is possible to expect an algorithm with $\cO(1/\sqrt{\varepsilon})$ upper complexity rather than $\cO(1/\varepsilon^2)$. 

\subsection{Motivation and main results}
Our research question can be motivated by the second remark above, i.e., a subclass might yield better lower complexity bounds. Note that the class of binary logistic regression problems is a subclass of the smooth convex problem class. Is it possible to design algorithms that targets solely on logistic regression, and performs better than the $\cO(1)(1/\sqrt{\varepsilon})$ complexity bounds for smooth convex optimization? Unfortunately, such question has not yet been answered in the literature. Although there had been lower complexity bounds $\cO(1)(1/\sqrt{\varepsilon})$ on smooth convex optimization (see Section \ref{sec:related_work}), the worst-case instance functions provided for smooth convex optimization are either based on convex quadratic functions \cite{nemirovski1992information,nesterov2004introductory,drori2017exact,woodworth2016tight,carmon2017lower,carmon2017lowerII} or smoothing (through infimal convolution) of maximum of affine functions \cite{guzman2015lower,diakonikolas2018lower}.

The above discussion is based on the traditional perspective of complexity analysis of convex optimization, namely, finding worst-case functions among the problem class and explore the performance limits of algorithms. 
It is important to point out that our research question can also be viewed from one other perspective. In data analysis practice, we will usually designing algorithms that are tailored for specific models. Consequently, we are interested at exploring the performance limit of algorithms with respect to the \emph{worst-case dataset}. From this perspective, our research question asks what the worst-case dataset that yields the worst performance of any deterministic first-order method. Note that the two aforementioned perspectives are equivalent; however, the latter one offers a more data-oriented argument.

In this paper, we describe some worst-case datasets of binary logistic regression problems, such that for any first-order methods, it requires at least $\cO(1)(1/\sqrt{\varepsilon})$ first-order oracle inquiries to obtain an $\varepsilon$-approximate solution. Such datasets can be used as certificates of optimal deterministic first-order algorithms for binary logistic regression. Also, from the perspective of traditional complexity analysis, our results also provide new worst-case functions for smooth convex optimization.

In Section \ref{sec:linear_span}, we describe the construction of a worst-case dataset for deterministic first-order method that satisfy a mild assumption (see Assumption \ref{assum:linear} below). In Section \ref{sec:generalLB}, we provide worst-case datasets for any given deterministic first-order method.

\section{Worst-Case Dataset Under Linear Span Assumption}
\label{sec:linear_span}

In this section, we make the following simple assumption regarding the iterates produced by a deterministic first-order method $\cM$:
\begin{assum}
	\label{assum:linear}
	The iterate sequence $\{\vc x_0,\vc x_1,\ldots\}$ produced by $\cM$ satisfies 
	\begin{align}
	\vc x_t\in\Span\{\nabla f(\vc x_0),\ldots,\nabla f(\vc x\tp)\},\ \forall t\ge 1.
	\end{align}
\end{assum}

\vgap

Recall that we have already made two assumptions in Section \ref{sec:intro} on $\cM$ without loss of generality, namely, that $\vc x_0=\vc 0$ and that $\vc x_t$ is both the inquiry point and the output of approximate solution. By the above assumption, the new iterate $\cM$ produces always lie inside the linear span of past gradients. Throughout this paper, we refer to Assumption \ref{assum:linear} as the {\em linear span assumption}. Such linear span assumption, in the first look, does not seem to be one that can be made without loss of generality. However, we would like to emphasize here that the purpose of introducing the linear span assumption is only for us to demonstrate the lower complexity bound derivation in a straightforward manner; we will show in Section \ref{sec:generalLB} that the linear span assumption can be removed, using the technique in the seminal work \cite{nemirovski1992information}.

\subsection{A Special Class of Datasets}
\label{sec:special_dataset}
We describe our construction of a special class of datasets for binary logistic regression. Such datasets will be used throughout this paper to construct worst-case datasets for binary logistic regression. Suppose that $\sigma>\zeta>0$ are two fixed real numbers. Given any positive integer $k$, denote
\begin{align}
	\label{eq:Wk}
	W_k:=\begin{pmatrix}
		 & & & -1 & 1
		\\
		 & & -1 & 1 &
		\\
		 & \iddots	& \iddots
		\\
		 -1 & 1 & 
		\\
		1 
	\end{pmatrix}\in\R^{k\times k}
\end{align}
and
\begin{align}
\label{eq:Akbk}
A_k:=\begin{pmatrix}
2\sigma W_k\\-2\zeta W_k\\ -2\sigma W_k \\ 2\zeta W_k  
\end{pmatrix}\in\R^{4k\times k},\ \vc b_k:=\begin{pmatrix}
	\vc 1_k\\\vc 1_k\\-\vc 1_k\\-\vc 1_k\\
\end{pmatrix}\in\R^{4k}.
\end{align}
We then denote functions $f_k:\R^k\to\R$ and $\Phi_k:\R^{k+1}\to\R$ by 
\begin{align}
\label{eq:fk}
f_k(\vc x):=h(A_k\vc x) - \vc b_k^\top A_k\vc x_k
\text{ and }
\Phi_k(\vc x, y):=h(A_k\vc x + y\vc 1_k) - \vc b_k^\top (A_k\vc x_k + y\vc 1_k)
.
\end{align}
Comparing \eqref{eq:fk} with previous descriptions of $l_{A,b}$ and $\Phi_{A,b}$ in \eqref{eq:general_problem} and \eqref{eq:general_l_homogeneous} respectively, $f_k$ and $\Phi_k$ are clearly binary logistic regression objective functions with data matrix $A_k$ and response vector $\vc b_k$: we are using logistic regression to train a classifier for two datasets whose entries have opposite signs. Recall that $\sigma>\zeta>0$; this assumption is to avoid duplicate data entries. Note that 
\begin{align}
	\|A_k \vc u\|^2 = & 8(\sigma^2+\zeta^2)\|W_k\vc u\|^2 
	\\
	= & 8(\sigma^2+\zeta^2)\left[\left(u^{(k)}-u^{(k-1)}\right)^2+\ldots+\left(u^{(2)} - u^{(1)}\right)^2 + \left(u^{(1)}\right)^2\right] 
	\\
	\le & 16(\sigma^2+\zeta^2)\left[\left(u^{(k)}\right)^2+\left(u^{(k-1)}\right)^2+\ldots+\left(u^{(2)}\right)^2 + \left(u^{(1)}\right)^2 + \left(u^{(1)}\right)^2\right]
	\\
	\le & 32(\sigma^2+\zeta^2)\|\vc u\|^2,
\end{align}
and consequently 
\begin{align}
	\label{eq:normAk}
	\|A_k\|\le 4\sqrt{2(\sigma^2+\zeta^2)}.
\end{align}

In the following lemma, we describe the optimal solutions that minimizes $f_k$ and $\Phi_k$ respectively. By the definition of $f_k$ in \eqref{eq:fk} and noting the convexity of binary logistic regression problems, it suffices to solve
\begin{align}
	\label{eq:grad_fk}
	\nabla f_k(\vc x) = A_k^\top \nabla h(A_k \vc x) - A_k^\top \vc b_k = 0.
\end{align} 
Noting the definition of $h$ in \eqref{eq:general_h}, we have
\begin{align}
	\label{eq:grad_h}
	\nabla h(\vc u) = \tanh\left(\frac{\vc u}{2}\right) := \left(\tanh\left(\frac{u^{(1)}}{2}\right), \ldots, \tanh\left(\frac{u^{(k)}}{2}\right)\right)^\top,\ \forall \vc u\in\R^k,\ \forall k.
\end{align}
Here $\tanh$ is the hyperbolic tangent function. 
Throughout this paper, we will slightly abuse the notation $\tanh(\vc u)$ and allow the scalar function $\tanh(\cdot)$ to be applied to any vector $\vc u$ component-wisely.

\begin{lem}
	\label{lem:xstar}
	For any $\sigma>\zeta>0$, there always exists $c>0$ that satisfies
	\begin{align}
		\label{eq:c}
		\sigma\tanh(\sigma c) + \zeta\tanh(\zeta c) = \sigma - \zeta.
	\end{align}
	Moreover,
	\begin{align}
		\label{eq:xstar}
		\vc x^*:=c(1,2,\ldots,k)^\top
	\end{align}
	is the unique optimal solution to problem
	\begin{align}
		\label{eq:fk_problem}
		f_k^*:=\min_{\vc x\in\R^k}f_k(\vc x)
	\end{align}
	with 
	\begin{align}
		\label{eq:fkstar}
		f_k^* = 8k\log 2 + 4k\left\{\log\cosh(\sigma c) + \log\cosh (\zeta c) - (\sigma - \zeta) c\right\}.
	\end{align}
	In addition, $(\vc x^*, 0)$ is the unique optimal solution to $\min_{\vc x\in\R^n, y\in\R}\Phi_k(\vc x, y)$.
\end{lem}

\begin{proof}
	Note that there always exists $c>0$ that satisfies \eqref{eq:c} since the function $r(c):= \sigma\tanh(\sigma c) + \zeta\tanh(\zeta c) - \sigma + \zeta$ is continuous with $r(0) = -\sigma + \zeta <0$ and $\lim_{c\to\infty}r(c) = 2\zeta >0$. 
	
	By the definitions of $W_k$ and $\vc x^*$ in \eqref{eq:Wk} and \eqref{eq:xstar}, we observe that $W_k\vc x^* = c\vc 1_k$ and $W_k\vc 1_k = \vc e_{k,k}$. Using this observation and the descriptions of $\nabla h$, $A_k$, and $\vc b_k$ in \eqref{eq:grad_h} and \eqref{eq:Akbk} respectively, and noting that $\tanh$ is an odd function and is assumed to apply to any vector component-wisely, we have
	\begin{align}
		\label{eq:grad_hAkx}
		\nabla h(A_k\vc x^*) = & \tanh\left[\frac{1}{2}\begin{pmatrix}
		2\sigma c \vc 1_k\\ -2\zeta c\vc 1_k\\ -2\sigma c \vc 1_k\\2\zeta c\vc 1_k 
		\end{pmatrix}\right] = \begin{pmatrix}
		\tanh(\sigma c) \vc 1_k\\ -\tanh(\zeta c)\vc 1_k\\ -\tanh(\sigma c) \vc 1_k\\\tanh(\zeta c)\vc 1_k 
		\end{pmatrix},
		\\
		A_k^\top \nabla h(A_k\vc x^*)  
		= & 4[\sigma \tanh(\sigma c) + \zeta \tanh(\zeta c)]\vc e_{k,k},
		\\
		\label{eq:AkTbk}
		A_k^\top \vc b_k = & 2(\sigma - \zeta + \sigma - \zeta)W_k\vc 1_k = 4(\sigma - \zeta)\vc e_{k,k}.
	\end{align}
	Using the above results, the description of $\nabla f$ in \eqref{eq:grad_fk}, and the relation \eqref{eq:c} that $c$ satisfies, 
	we have $\nabla f_k(\vc x^*)=0$. 
	Noting that binary logistic loss functions are strictly convex, we conclude that $\vc x^*$ is the unique minimizer of $f_k$. Recalling the definition of $h$ in \eqref{eq:general_h}, noting that $\cosh$ is an even function, and using the computation of $A_k^\top \vc b_k$ in \eqref{eq:AkTbk}, we have
	\begin{align*}
		f_k^* = & f_k(\vc x^*) = h(A_k\vc x^*) - \left(\vc x^*\right)^\top A_k^\top \vc b_k 
		= h(A_k\vc x^*) - 4k(\sigma - \zeta)c
		\\
		= & 2k\left\{\log\left[2\cosh(\sigma c)\right]+\log\left[2\cosh(-\zeta c)\right] + \log\left[2\cosh(-\sigma c)\right]+\log\left[2\cosh(\zeta c)\right]\right\} - 4k(\sigma-\zeta)c
		\\
		= & 8k\log 2 + 4k\left\{\log\cosh(\sigma c) + \log\cosh (\zeta c) - (\sigma - \zeta) c\right\}.
	\end{align*}

	Furthermore, by the descriptions of $\vc b_k$ and $\nabla h(A_k\vc x^*)$ in \eqref{eq:Akbk} and \eqref{eq:grad_hAkx} respectively, computing the partial derivative of $\Phi$ in \eqref{eq:fk} with respective to $y$ at $0$, we have
	\begin{align}
	& \left.\frac{\partial}{\partial y}\right|_{y=0}\Phi_k(\vc x^*, y) = \vc 1_k^\top \nabla h(A_k\vc x^*) - \vc b_k^\top \vc 1_k = 0.
	\end{align}
	Noting also that $\nabla_x \Phi_k(\vc x^*, 0) = \nabla f_k(\vc x^*) = 0$, we conclude that $(\vc x^*,0)$ is the unique minimizer of the strictly convex binary logistic loss $\Phi_k(\vc x, y)$.

\end{proof}

\subsection{Lower Complexity Bound Under Linear Span Assumption}
In this section, we study the lower complexity bound of deterministic first-order methods for solving the logistic regression problem \eqref{eq:fk_problem}, under the linear assumption described in Assumption \ref{assum:linear}.

\begin{lem}
	\label{lem:KJ}
	Suppose that $k$ and $t$ are fixed positive integers such that $t\le k$. Define
	\begin{align}
		\label{eq:Kt}
		\cK_{t,k}:=\Span\{\vc e_{k-t+1,k},\ldots,\vc e_{k,k}\},\ \forall k, \forall 1\le t\le k.
	\end{align}
	Then for all $\vc x\in\cK_{t,k}$, we have $A_k\vc x, \nabla h(A_k\vc x)\in \cJ_{t,k}$ and $A_k^\top \nabla h(A_k\vc x), \nabla f_k(\vc x)\in\cK_{t+1,k}$, where
	\begin{align}
		\label{eq:Jt}
		\cJ_{t,k}:=\Span\{\vc e_{1,4k},\ldots,\vc e_{t,4k},\vc e_{k+1,4k},\ldots,\vc e_{k+t,4k}, \vc e_{2k+1,4k},\ldots,\vc e_{2k+t,4k}, \vc e_{3k+1,4k},\ldots,\vc e_{3k+t,4k}\}.
	\end{align}
	Moreover,
	\begin{align}
		\label{eq:fkftstar}
		\min_{\vc x\in\cK_{t,k}}f_k(\vc x) = 8(k-t)log 2 + \min_{\vc u\in\R^t}f_t(\vc u).
	\end{align}
\end{lem}
\begin{proof}
	Fix $\vc x\in \cK_{t,k}$. By \eqref{eq:Kt} we have $\vc x^\top = (\vc 0_{k-t}^\top, \vc u^\top)^\top$ for some $\vc u\in\R^t$. Thus by the definition of $W_k$ in \eqref{eq:Wk},
	\begin{align}
		\label{eq:Wku}
		W_k\vc x = \left(\begin{array}{c;{2pt/2pt}c}
		\ \ \ \ \ _{-1} & W_t \\ \hdashline[2pt/2pt]
		W_{k-t} & 
		\end{array}\right)
		\begin{pmatrix}
			\vc 0_{k-t} \\ \vc u
		\end{pmatrix} = \begin{pmatrix}
			W_t\vc u\\ \vc 0_{k-t}
		\end{pmatrix}.
	\end{align}
	Using the above result, the descriptions of $A_k$ and $\nabla h$ in \eqref{eq:Akbk} and \eqref{eq:grad_h} respectively, and the definition of $\cJ_{t,k}$ in \eqref{eq:Jt}, we have
	\begin{align}
		\label{eq:tmp1}
		A_k\vc x = \begin{pmatrix}
			2\sigma W_t \vc u \\ \vc 0_{k-t} \\ -2\zeta W_t \vc u\\ \vc 0_{k-t} \\ -2\sigma W_t \vc u \\ \vc 0_{k-t} \\ 2\zeta W_t \vc u\\ \vc 0_{k-t}
		\end{pmatrix}\in \cJ_{t,k},\ \nabla h(A_k\vc x) = \begin{pmatrix}
		\tanh(\sigma W_t \vc u) \\ \vc 0_{k-t} \\ -\tanh(\zeta W_t \vc u)\\ \vc 0_{k-t} \\ -\tanh(\sigma W_t \vc u) \\ \vc 0_{k-t} \\ \tanh(\zeta W_t \vc u)\\ \vc 0_{k-t}
		\end{pmatrix}\in \cJ_{t,k}.
	\end{align}
	Also, note that for all $\vc v\in\R^t$, the definition of $W_k$ in \eqref{eq:Wk} results in 
	\begin{align}
		\label{eq:WkTv}
		W_k^\top \begin{pmatrix}
			\vc v\\ \vc 0_{k-t}
		\end{pmatrix} = W_k \begin{pmatrix}
		\vc v\\ \vc 0_{k-t}
		\end{pmatrix} = \left(\begin{array}{c;{2pt/2pt}c}
		\ \ \ \ \ _{-1} & W_{k-t} \\ \hdashline[2pt/2pt]
		W_t & 
		\end{array}\right)\begin{pmatrix}
		\vc v\\ \vc 0_{k-t}
		\end{pmatrix} = \begin{pmatrix}
			\vc 0_{k-t-1} \\ -\vc v^{(t)} \\ W_{t}\vc v
		\end{pmatrix} \in\cK_{t+1,k}.
	\end{align}
	Combining \eqref{eq:tmp1} and \eqref{eq:WkTv}, and using the definition of $A_k$ in \eqref{eq:Akbk} we have
	\begin{align}
		A_k^\top \nabla h(A_k\vc x) = &
			2\sigma W_k^\top \begin{pmatrix}
			\tanh(\sigma W_t \vc u) \\ \vc 0_{k-t}
			\end{pmatrix} - 2\zeta W_k^\top \begin{pmatrix}
			-\tanh(-\zeta W_t \vc u) \\ \vc 0_{k-t}
			\end{pmatrix} 
			\\
			& - 2\sigma W_k^\top \begin{pmatrix}
			-\tanh(\sigma W_t \vc u) \\ \vc 0_{k-t}
			\end{pmatrix} + 2\zeta W_k^\top \begin{pmatrix}
			\tanh(\zeta W_t \vc u) \\ \vc 0_{k-t}
			\end{pmatrix} 
			\\
			\in & \cK_{t+1,k}.
	\end{align}
	Using the above result and noting the value of $A_k^\top \vc b_k$ in \eqref{eq:AkTbk}, we conclude that
	\begin{align}
		\nabla f_k(\vc x) = A_k^\top \nabla h(A_k\vc x) - A_k^\top \vc b_k \in \cK_{t+1,k}.
	\end{align}
	
	To finish the proof it suffices to prove \eqref{eq:fkftstar}. By the definition of $h$ in \eqref{eq:general_h}, the computations in \eqref{eq:tmp1}, and the setting $\vc x^\top = (\vc 0_{k-t}^\top, \vc u^\top)^\top$ we have
	\begin{align}
		h(A_k\vc x) = & \sumt[t]2\log(2\cosh(\sigma W_tu)^{(i)}) + (k-t)\cdot 2\log(2\cosh(0))
		\\ 
		& + \sumt[t]2\log(2\cosh(-\zeta W_tu)^{(i)}) + (k-t)\cdot 2\log(2\cosh(0))
		\\
		& + \sumt[t]2\log(2\cosh(-\sigma W_tu)^{(i)}) + (k-t)\cdot 2\log(2\cosh(0))
		\\
		& + \sumt[t]2\log(2\cosh(\zeta W_tu)^{(i)}) + (k-t)\cdot 2\log(2\cosh(0))
		\\
		= & 8(k-t)\log 2 + h(A_t\vc u).
	\end{align}
	Also, noting that $\vc x^\top = (\vc 0_{k-1}^\top, \vc u^\top)^\top$, by the description of $A_k^\top\vc b_k$ in \eqref{eq:AkTbk} we have
	\begin{align}
		\vc b_k^\top A_k \vc x = 4(\sigma - \zeta) u^{(t)} = \vc b_t^\top A_t \vc u.
	\end{align}
	Hence we conclude from the definition of $f_k(\vc x)$ in \eqref{eq:fk} that $f_k(\vc x) = 8(k-t)\log 2+ f_t(\vc u)$, and thus \eqref{eq:fkftstar} holds.
\end{proof}

As an immediate consequence of the above lemma, in the following we show that the linear span assumption of a first-order method $\cM$ will lead to $\vc x_t\in\cK_{t,k}$ when minimizing $f_k(\vc x)$.

\begin{lem}
	\label{lem:xt_in_Kt}
	Suppose that $\cM$ is any deterministic first-order method that satisfies Assumption \ref{assum:linear}. When $\cM$ is applied to minimize $f_k(\vc x)$ in \eqref{eq:fk}, we have
	$\vc x_t\in\cK_{t,k}$
	for all $1\le t\le k$.
\end{lem}

\begin{proof}
	We prove the $t=1$ case first. By Assumption \ref{assum:linear},  $\vc x_1\in\Span\{\nabla f_k(\vc x_0)\}$. Recalling the assumption that $\vc x_0=\vc 0$, we have $\nabla f_k(\vc x_0) = \nabla f_k(\vc 0) = -A_k^\top \vc b_k$, and by the value of $A_k^\top \vc b_k$ in \eqref{eq:AkTbk} we have $\nabla f_k(\vc x_0)\in\Span\{\vc e_{k,k}\}$. Noting the definition of $\cK_{t,k}$ in \eqref{eq:Kt} we have $\vc x_1\in\cK_{1,k}$. 
	
	Let us use induction and assume that $\vc x_i\in\cK_{i,k}$ for all $1\le i\le s<k$. By Lemma \ref{lem:KJ}, we have $\nabla f_k(\vc x_i)\in\cK_{i+1,k}$ for all $s$. Noting Assumption \ref{assum:linear} we have 
	$$\vc x_{s+1}\in\Span\{\nabla f_k(\vc x_0),\ldots,\nabla f_k(\vc x_s)\} \subseteq \cK_{s+1,k}. $$
	Hence the induction is complete and we conclude that $\vc x_t\in\cK_{t,k}$
	for all $1\le t\le k$.
\end{proof}

By the description of $f_k^*$ in \eqref{eq:fkstar}, the relation \eqref{eq:fkftstar}, and Lemma \ref{lem:xt_in_Kt}, we conclude that the error of iterate $\vc x_t$ in terms of objective function value can be lower bounded by
\begin{align}
\label{eq:obj_error}
\begin{aligned}
	f_k(\vc x_t) - f_k^* \ge & \min_{\vc x\in\cK_{t,k}}f_k(\vc x) - f_k^* = 8(k-t)\log 2 + f_t^* - f_k^*
	\\
	= & 4(k-t)\left[(\sigma - \zeta) c - \log\cosh(\sigma c) - \log\cosh (\zeta c)\right].
\end{aligned}
\end{align}
In the following lemma, we provide a simplification of the above lower bound:

\begin{lem}
	\label{lem:Delta_est}
	For any real numbers $\sigma$ and $\zeta$ that satisfy $2\zeta>\sigma>\zeta>0$, we have
	\begin{align}
	(\sigma - \zeta)c - \log\cosh(\sigma c) - \log\cosh (\zeta c) \ge c^2\sigma^2C(\sigma/\zeta),
	\end{align}
	where $C(\sigma/\zeta)$ is a universal constant that depends only on the ratio $\sigma/\zeta$. In particular, When $\sigma/\zeta = 1.3$, we have
	\begin{align}
	\label{eq:C_special}
	C(1.3) > \frac{1}{2}.
	\end{align}
\end{lem}

\begin{proof}
	By checking its derivative it is easy to verify that the function $c\mapsto c\tanh(c)$ is increasing when $c>0$. Hence we have
	\begin{align}
	\zeta c \tanh(\zeta c) \le \sigma c \tanh(\sigma c),\text{ i.e., }\zeta  \tanh(\zeta c) \le \sigma  \tanh(\sigma c).
	\end{align}
	Applying the above relation to \eqref{eq:c}, we have
	\begin{align}
	2\zeta\tanh(\zeta c) \le \sigma - \zeta \le 2\sigma\tanh(\sigma c).
	\end{align}
	Since $\tanh$ is an increasing function, we have from the above inequality that
	\begin{align}
	\label{eq:c_bounds}
	c\in [c_{lb}, c_{ub}],\text{ where }c_{lb}:=\frac{1}{\sigma}\atanh\left(\frac{1}{2}-\frac{\zeta}{2\sigma}\right) \text{ and }c_{ub}:=\frac{1}{\zeta}\atanh\left(\frac{\sigma}{2\zeta} - \frac{1}{2}\right)
	\end{align}
	in which $c_{lb},c_{ub}>0$ are well-defined real numbers under the assumption that $2\zeta>\sigma>\zeta>0$.
	Using the above result, the definition of $c$ in \eqref{eq:c}, and noting that the function $c\mapsto c\tanh(c) - \log\cosh c$ is increasing when $c>0$ (by checking its derivative), we have
	\begin{align}
	& (\sigma - \zeta)c - \log\cosh(\sigma c) - \log\cosh (\zeta c)
	\\
	= & c^2\sigma^2\frac{1}{c^2\sigma^2}\left[\sigma c\tanh(\sigma c) - \log\cosh(\sigma c) + \zeta c\tanh(\zeta c) - \log\cosh (\zeta c)\right]
	\\
	\ge & c^2\sigma^2C
	\end{align}
	where
	\begin{align}
	C:=\frac{1}{c_{ub}^2\sigma^2}\left[\sigma c_{lb}\tanh(\sigma c_{lb}) - \log\cosh(\sigma c_{lb}) + \zeta c_{lb}\tanh(\zeta c_{lb}) - \log\cosh (\zeta c_{lb})\right].
	\end{align}
	Noting \eqref{eq:c_bounds}, we can observe that the above constant $C$ depends only on the ratio $\sigma/\zeta$. The result \eqref{eq:C_special} can then be computed numerically.
\end{proof}

We are now ready to state a lower complexity bound of deterministic first-order methods under the linear span assumption.

\begin{thm}
	Suppose that $\cM$ is any deterministic first-order method that satisfies the linear span assumption in Assumption \ref{assum:linear}. Given any iteration number $T$, there always exist data matrix $A\in\R^{N\times n}$ and response vector $\vc b\in\{-1,1\}^N$, where $n=2T$ and $N=8T$, such that the $T$-th approximate solution $\vc x_T$ generated by $\cM$ on minimizing the binary logistic loss function $l_{A,b}$ in \eqref{eq:general_l_homogeneous} satisfies
	\begin{align}
		\label{eq:lb_linear_span}
		\begin{aligned}
		l_{A,b}(\vc x_T) - l_{A,b}^*> & \frac{3\|A\|^2\|\vc x_0 - \vc x^*\|^2}{32(2T+1)(4T+1)},
		\\
		\|\vc x_T - \vc x^*\|^2 > & \frac{1}{8}\|\vc x_0 - \vc x^*\|^2,
		\end{aligned}
	\end{align}
	where $x^*$ is the minimizer of $f$.
\end{thm}

\begin{proof}
	Let us fix any $\zeta>0$ and set $\sigma = 1.3\zeta$ in the definition of $A_k$ in \eqref{eq:Akbk}. By \eqref{eq:normAk} we have
	\begin{align}
		\label{eq:normAksigma}
		\|A_k\|\le 4\sqrt{2\sigma^2 + 2(\sigma/1.3)^2} < 8\sigma.
	\end{align} 
	Let us apply $\cM$ to minimize $f_{k}$ defined in \eqref{eq:fk} where $k=2T$. Recall that $\cM$ starts at $\vc x_0 = 0$, and that the minimizer $\vc x^*$ in \eqref{eq:xstar} satisfies
	\begin{align}
		\label{eq:xstarnorm}
		\|\vc x_0-\vc x^*\|^2 = c^2\sumt[k]i^2 = \frac{c^2}{6}k(k+1)(2k+1).
	\end{align}
	
	By Lemmas \ref{lem:xt_in_Kt}, \ref{lem:Delta_est}, the lower bound estimate \eqref{eq:obj_error}, and noting that $\sigma>\zeta$, we have $\vc x_t\in\cK_{t,k}$ and 
	\begin{align}
		f_k(\vc x_t) - f_k^* \ge 2(k-t)c^2\sigma^2,\ \forall t\le k.
	\end{align}
	Applying \eqref{eq:normAksigma} and \eqref{eq:xstarnorm}, the above relation becomes
	\begin{align}
		\label{eq:tmp2}
		f_k(\vc x_t) - f_k^* > \frac{3(k-t)\|A_k\|^2\|\vc x_0-\vc x^*\|^2}{16k(k+1)(2k+1)}.
	\end{align}
	Also, since $\vc x_t\in\cK_{t,k}$, by the definition of $\cK_t$ in \eqref{eq:Kt} we have $x_t^{(1)} = \ldots = x_t^{(k-t)} = 0$. Noting the description of $\vc x^*$ in \eqref{eq:xstar} and focusing on the difference between $\vc x_t$ and $\vc x^*$ in the first $(k-t)$ components, we have
	\begin{align}
		\label{eq:tmp3}
		\|\vc x_t - \vc x^*\|^2 \ge c^2\sumt[k-t]i^2 = \frac{c^2}{6}(k-t)(k-t+1)(2k-2t+1).
	\end{align}
	Specially, setting $t=T$ and recalling that $k=2T$, \eqref{eq:tmp2} becomes
	\begin{align}
		f_k(\vc x_T) - f_k^*> \frac{3\|A_k\|^2\|x_0 - x^*\|^2}{32(2T+1)(4T+1)},
	\end{align}
	and \eqref{eq:xstarnorm} and \eqref{eq:tmp3} imply that
	\begin{align}
		\|\vc x_T - \vc x^*\|^2 \ge \frac{c^2}{6}T(T+1)(2T+1) > \frac{c^2}{48}\cdot 2T(2T+1)(4T+1) = \frac{1}{8}\|\vc x_0 - \vc x^*\|^2.
	\end{align}
	We conclude \eqref{eq:lb_linear_span} from the above two results by setting $A:=A_k\in\R^{8T\times 2T}$, $\vc b:=\vc b_k\in\R^{8T}$ and noting the equivalence between $l_{A,b}$ in \eqref{eq:general_l_homogeneous} and $f_k$ in the above derivation. 

\end{proof}

\section{Lower Complexity Bound for General Deterministic First-Order Methods}
\label{sec:generalLB}

In this section, we extend the lower complexity bound to general deterministic first-order methods. The derivation is based on the concept of orthogonal invariance in the seminal work \cite{nemirovski1992information}, and is organized in a similar way as in \cite{ouyang2018lower}.
Note that we can also use the concept of zero-respecting algorithms in \cite{carmon2017lower,carmon2017lowerII} to finish the proof. 

We will use the following technical lemma which is proved in \cite{ouyang2018lower} (see Lemma 3.1 within). 
\begin{lemma}
	\label{lem:ortho_map}
	Let $\cX\subsetneq \bar \cX\subseteq\R^p$ be two linear subspaces. Then for any $\bar {\vc x}\in \R^p$, there exists an orthogonal matrix $V\in \R^{p\times p}$ such that
	\begin{align}
	\label{eq:lemU}
	V\vc x = \vc x,\ \forall \vc x\in \cX,\text{ and }V\bar {\vc x}\in\bar \cX.
	\end{align}
\end{lemma}

\begin{pro}
	\label{pro:rotation_trick}
	For any $A_k$ and $b_k$ in the form of \eqref{eq:Akbk}, any deterministic first-order method $\cM$, and any $t\le (k-3)/2$, there exists an orthogonal matrix $U_t\in\R^{k\times k}$ that satisfy the following:
	\begin{enumerate}
		\item $U_tA_k^\top \vc b_k = A_k^\top \vc b_k$;
		\item When $\cM$ is applied to minimize the binary logistic regression loss function $l_{A_kU_t,\vc b_k}$ defined in \eqref{eq:general_l_homogeneous}, its iterates $\vc x_0,\ldots,\vc x_t$ satisfy
		\begin{align}
			\vc x_i\in U_t^\top \cK_{2i+1},\ \forall i =0,\ldots,t.
		\end{align}
	\end{enumerate}
\end{pro}
\begin{proof}
	Let us fix $A_k$, $b_k$ and the method $\cM$. Throughout this proof, we will use the notation
	\begin{align}
		\cU:=\Set{V\in\R^{k\times k}|V\text{ is orthogonal and } VA_k^\top \vc b_k = A_k^\top \vc b_k}.
	\end{align}
	
	We conduct the proof by induction. 	
	The case when $t=0$ is trivial by setting $U_0$ to be the identity matrix. 
	Let us assume that the proposition is true when $t=s-1< (k-1)/2$. By the induction hypothesis there exists $U_{s-1}\in\cU$ such that when $\cM$ is applied to minimize $l_{A_kU_{s-1},\vc b_k}$, its iterates satisfy
	\begin{align}
		\label{eq:ind_hypo}
		\vc x_i\in U_{s-1}^\top \cK_{2i+3,k},\ \forall i=0,\ldots,s-1.
	\end{align}
	Suppose that $\vc x_s$ is the next iterate. To prove the case when $t=s$, let us start by finding an orthogonal matrix $U_s\in\cU$. 
	Noting that $s< (k-1)/2$, from the definition of $\cK_{t,k}$ in \eqref{eq:Kt} we have 
	\begin{align}
	\label{eq:KtSubset}
	\cK_{1,k}\subsetneq \cK_{2,k}\subsetneq \ldots \subsetneq \cK_{2s+1,k}.
	\end{align}
	Thus $U_{s-1}^\top\cK_{2s}\subsetneq U_{s-1}^\top\cK_{2s+1}$, and by Lemma \ref{lem:ortho_map} there exists orthogonal matrix $V$ such that
	\begin{align}
	\label{eq:V}
	V\vc x = \vc x,\ \forall x\in U_{s-1}^\top\cK_{2s}, \text{ and }V\vc x_s\in U_{s-1}^\top\cK_{2s+1}.
	\end{align}
	Let us define 
	\begin{align}
	\label{eq:Us}
	U_s:=U_{s-1}V.
	\end{align}
	Noting the descriptions of $A_k^\top\vc b_k$ and $\cK_{1,k}$ in \eqref{eq:AkTbk} and \eqref{eq:Kt} respectively, we observe that $A_k^\top \vc b_k\in\cK_{1,k}\subset \cK_{2s,k}$. Using such observation, by \eqref{eq:V}, \eqref{eq:Us}, and the induction hypothesis $U_{s-1}\in\cU$, we have 
	$
		U_s^\top A_k^\top \vc b_k = V^\top U_{s-1}^\top A_k^\top \vc b_k = A_k^\top \vc b_k,
	$
	hence $U_s\in\cU$. Also, from \eqref{eq:KtSubset} we have $U_{s-1}^\top \cK_{2i+1,k}\subset U_{s-1}^\top \cK_{2s}$ for all $i=0,\ldots,s-1$. Consequently by \eqref{eq:V} and \eqref{eq:Us} we have
	\begin{align}
	\label{eq:tmp4}
	U_s^\top \cK_{2i+1,k} = V^\top U_{s-1}^\top \cK_{2i+1,k} = U_{s-1}^\top \cK_{2i+1,k},\ \forall i=1,\ldots,s-1.
	\end{align}
	Applying the above relation to \eqref{eq:ind_hypo} and also noting $\vc x_s\in U_s^\top \cK_{2s+1}$ from \eqref{eq:V} and \eqref{eq:Us}, we obtain
	\begin{align}
	\label{eq:Us_xi}
	\vc x_i\in U_s^\top \cK_{2i+1,k},\ \forall i=0,\ldots,s.
	\end{align}
	
	Let us apply $\cM$ to minimize $l_{A_kU_s,\vc b_k}$. We will prove that its first $s+1$ iterates are exactly $\vc x_0,\ldots,\vc x_s$ (the ones computed when $\cM$ is applied to $l_{A_kU_{s-1},\vc b_k}$). Indeed, we can make the following observation: if
	\begin{align}
		\label{eq:indistinguishable_oracle}
		l_{A_kU_s,\vc b_k}(\vc x) = l_{A_kU_{s-1},\vc b_k}(\vc x)\text{ and }\nabla l_{A_kU_s,\vc b_k}(\vc x) = \nabla l_{A_kU_{s-1},\vc b_k}(\vc x),\ \forall \vc x\in U_s^\top \cK_{2s-1,k},
	\end{align}
	then by \eqref{eq:Us_xi} and the oracle assumption \eqref{eq:xOiter}, $\cM$ would obtain exactly the same first-order information at $\vc x_0,\ldots,\vc x_{s-1}\in U_s^\top \cK_{2s-1,k}$ from the first-order oracle when minimizing either $l_{A_kU_s,\vc b_k}$ or $l_{A_kU_{s-1},\vc b_k}$. Therefore, if \eqref{eq:indistinguishable_oracle} holds, then $\cM$ produces exactly the same iterates $\vc x_0,\ldots,\vc x_s$ when minimizing either $l_{A_kU_s,\vc b_k}$ or $l_{A_kU_{s-1},\vc b_k}$. Consequently, noting that $U_s\in \cU$ and \eqref{eq:Us_xi} we obtain the results of the $t=s$ case by choosing $U=U_s$ and complete the induction.
	
	To finish the induction proof it suffices to prove \eqref{eq:indistinguishable_oracle}. 
	Let us fix any $\vc x\in U_s^\top \cK_{2s-1,k}$. By \eqref{eq:V} and \eqref{eq:tmp4} we have $\vc x\in U_{s-1}^\top \cK_{2s-1,k}$. Noting \eqref{eq:V} and that $U_{s-1},U_s\in\cU$, we obtain the following relations:
	\begin{align}
		\label{eq:tmp5}
		U_{s}\vc x = U_{s-1}V\vc x = U_{s-1}\vc x \in \cK_{2s-1,k}\text{ and }U_{s}^\top A_k^\top \vc b_k = A_k^\top \vc b_k = U_{s-1}^\top A_k^\top \vc b_k.
	\end{align}
	Moreover, noting that $U_{s-1}\vc x\in \cK_{2s-1,k}$, applying Lemma \ref{lem:KJ} we have $A_k^\top \nabla h(A_kU_{s-1}\vc x)\in\cK_{2s,k}$, and hence by \eqref{eq:V} we observe that $V^\top U_{s-1}^\top A_k^\top \nabla h(A_kU_{s-1}\vc x) = U_{s-1}^\top A_k^\top \nabla h(A_kU_{s-1}\vc x)$. Using such observation, recalling the definition of $l_{A,\vc b}$ in \eqref{eq:general_l_homogeneous}, and noting the relations in \eqref{eq:tmp5}, we conclude that
	\begin{align}
	l_{A_kU_s,\vc b_k}(\vc x) = & h(A_kU_s\vc x) - \vc x^\top U_s^\top A_k^\top \vc b_k = h(A_kU_{s-1}\vc x) - \vc x^\top U_{s-1}^\top A_k^\top \vc b_k = \vc l_{A_kU_{s-1},\vc b_k}(\vc x),
	\\
	\nabla l_{A_kU_s,\vc b_k}(\vc x) = & U_s^\top A_k^\top \nabla h(A_kU_s\vc x) - U_s^\top A_k^\top \vc b_k = V^\top U_{s-1}^\top A_k^\top \nabla h(A_kU_{s-1}\vc x) - U_{s-1}^\top A_k^\top \vc b_k
	\\
	= & U_{s-1}^\top A_k^\top \nabla h(A_kU_{s-1}\vc x) - U_{s-1}^\top A_k^\top \vc b_k = \nabla l_{A_kU_{s-1},\vc b_k}(\vc x).
	\end{align}
	Hence \eqref{eq:indistinguishable_oracle} is proved.
	
\end{proof}

\begin{thm}
	Suppose that $\cM$ is any deterministic first-order method. Given any iteration number $T$, there always exists data matrix $A\in\R^{N\times n}$ and $b\in\R^N$, where $n=4T+2$ and $N=16T+8$, such that the $T$-th approximate solution $\vc x_T$ generated by $\cM$ on minimizing the binary logistic regression loss function $l_{A,b}$ in \eqref{eq:general_l_homogeneous} satisfies
\begin{align}
	\label{eq:lb_general}
	\begin{aligned}
	l_{A,\vc b}(\vc x_T) - l_{A,\vc b}^*\le & \frac{3\|A\|^2\|\vc x_0 - \vc z^*\|^2}{32(4T+3)(8T+5)}
	\\
	\|\vc x_T - \vc z^*\|^2> & \frac{1}{8}\|\vc x_0 - \vc z^*\|^2,
	\end{aligned}
\end{align}	
where $\vc z^*$ is the minimizer of $l_{A,b}$.
\end{thm}

\begin{proof}
	Let us fix any $\zeta>0$ and set $\sigma = 1.3\zeta$ in the definition of $A_k$ in \eqref{eq:Akbk}, in which we set $k=4T+2$. Note that the norm of $A_k$ satisfies \eqref{eq:normAksigma}.	
	Applying Proposition \ref{pro:rotation_trick} to $A_k$, $b_k$, and $\cM$ with $t=T$, we obtain the following result: there exists an orthogonal matrix $U:=U_T$ that satisfies $U^\top A_k^\top \vc b_k = A_k^\top \vc b_k$, such that when $\cM$ is applied to minimize $l_{A_kU,\vc b_k}$, its iterates $\vc x_i$ satisfies $\vc x_i\in U^\top \cK_{2i+1,k}$ for all $0\le i\le T$. Note that in this result we have
	\begin{align}
		l_{A_kU,\vc b_k}(\vc x_T)\ge & \min_{\vc x\in U^\top \cK_{2T+1,k}}l_{A_kU,\vc b_k}(\vc x) = \min_{\vc x\in U^\top \cK_{2T+1,k}}h(A_kU\vc x) - \vc x^\top U^\top A_k^\top \vc b_k 
		\\
		=& \min_{\vc x\in\cK_{2T+1,k}}h(A_k\vc x) - \vc x^\top A_k^\top \vc b_k = \min_{\vc x\in\cK_{2T+1,k}}f_k(\vc x)\text{ and }
		\\
		\label{eq:tmp6}
		l_{A_kU,\vc b_k}^*= & \min_{\vc x\in\R^k}l_{A_kU,\vc b_k}(\vc x) = \min_{\vc x\in \R^k}h(A_kU\vc x) - \vc x^\top U^\top A_k^\top \vc b_k
		\\
		= &   \min_{\vc x\in \R^k}h(A_k\vc x) - \vc x^\top A_k^\top \vc b_k  = \min_{\vc x\in\R^k}f_k(\vc x) = f_k^*.
	\end{align}
	Here we use the definition of $f_k$ in \eqref{eq:fk}. Consequently,
	\begin{align}
		\label{eq:tmp9}
		\begin{aligned}
		& l_{A_kU,\vc b_k}(\vc x_T) - l_{A_kU,\vc b_k}^*
		\ge  \min_{\vc x\in\cK_{2T+1,k}}f_k(\vc x) - f_k^*.
		\end{aligned}
	\end{align}	
	Note from \eqref{eq:tmp6} above that the minimizer $\vc z^*$ of $l_{A_kU,\vc b_k}(\vc x)$ satisfies $\vc z^* = U^\top\vc x^*$, where $x^*$ is the minimizer of $f_k$ defined in \eqref{eq:xstar}. Since $\vc x_T\in U^\top \cK_{2T+1,k}$, we have
	\begin{align}
		\|\vc x_T - \vc z^*\|^2 \ge & \max_{\vc x\in\cK_{2T+1}}\|\vc x - \vc x^*\|^2 
		\\
		\ge & c^2\sumt[k-2T-1]i^2 = \frac{c^2}{6}(k-2T-1)(k-2T)(2k-4T-1) 
		\\
		= & \frac{c^2}{6}(2T+1)(2T+2)(4T+1).
	\end{align}
	Here the last equality is since we set $k=4T+2$. Also, recalling that $\cM$ starts at $\vc x_0=0$ we have
	\begin{align}
		\label{eq:tmp7}
		\|\vc x_0 - \vc z^*\|^2 = \|\vc x^*\|^2 = & c^2\sumt[4T+2]i^2 = \frac{c^2}{6}(4T+2)(4T+3)(8T+5).
	\end{align}
	Summarizing the above two relations we have
	\begin{align}
		\label{eq:tmp8}
		\|\vc x_T - \vc z^*\|^2> & \frac{1}{8}\|\vc x_0 - \vc z^*\|^2.
	\end{align}
	Furthermore, applying \eqref{eq:tmp8}, Lemma \ref{lem:Delta_est}, and the estimate of lower bound in \eqref{eq:obj_error} to  \eqref{eq:tmp9}, we have	
	\begin{align}
		& l_{A_kU,\vc b_k}(\vc x_T) - l_{A_kU,\vc b_k}^*
		\\
		\ge & 4(k - 2T-1)\left[(\sigma - \zeta) c - \log\cosh(\sigma c) - \log\cosh (\zeta c)\right]
		\\
		\ge & 2(k-2T-1)c^2\sigma^2
		\\
		= & \frac{6*(k-2T-1)\sigma^2\|\vc x_0 - \vc z^*\|^2}{(2T+1)(4T+3)(8T+5)}.
	\end{align}
	Applying the estimate of $\|A_k\|$ in \eqref{eq:normAk} to the above, and recalling that $k=4T+2$, we obtain
	\begin{align}
		\label{eq:tmp10}
		& l_{A_kU,\vc b_k}(\vc x_T) - l_{A_kU,\vc b_k}^*\le \frac{3\|A\|^2\|\vc x_0 - \vc z^*\|^2}{32(4T+3)(8T+5)}.
	\end{align}
	
	By setting $A:=A_kU\in\R^{(16T+8)\times (4T+2)}$ and $b:=\vc b_k\in\R^{(16T+8)}$, we conclude the proof from \eqref{eq:tmp9} and \eqref{eq:tmp8}.
\end{proof}

\section{Concluding Remarks}
In this paper, we describe some worst-case datasets for deterministic first-order methods on solving binary logistic regression. The binary logistic regression functions with our worst-case datasets can also server as new worst-case function instances among the class of smooth convex optimization problems.

It should be noted that our description of $A_k$ and $\vc b_k$ in \eqref{eq:Akbk} are designed so that the optimal intercept of binary logistic regression is $0$. If we are focusing only on homogeneous linear predictor case without requiring the optimal intercept to be $0$, an easier dataset can be designed by simply setting
\begin{align}
A_k:=\begin{pmatrix}
2\sigma W_k\\2\zeta W_k  
\end{pmatrix}\in\R^{2k\times k},\ \vc b_k:=\begin{pmatrix}
\vc 1_k\\-\vc 1_k\\
\end{pmatrix}\in\R^{2k}
\end{align}
and follow the derivations in Sections \ref{sec:linear_span} and \ref{sec:generalLB}.


%
%
%
%
%
%

\vskip 0.2in
\bibliography{yuyuan}

\begin{thebibliography}{12}
\providecommand{\natexlab}[1]{#1}
\providecommand{\url}[1]{\texttt{#1}}
\expandafter\ifx\csname urlstyle\endcsname\relax
  \providecommand{\doi}[1]{doi: #1}\else
  \providecommand{\doi}{doi: \begingroup \urlstyle{rm}\Url}\fi

\bibitem[Bach(2010)]{bach2010self}
Francis Bach.
\newblock Self-concordant analysis for logistic regression.
\newblock \emph{Electronic Journal of Statistics}, 4:\penalty0 384--414, 2010.

\bibitem[Carmon et~al.(2017{\natexlab{a}})Carmon, Duchi, Hinder, and
  Sidford]{carmon2017lower}
Yair Carmon, John~C Duchi, Oliver Hinder, and Aaron Sidford.
\newblock Lower bounds for finding stationary points {I}.
\newblock \emph{arXiv preprint arXiv:1710.11606}, 2017{\natexlab{a}}.

\bibitem[Carmon et~al.(2017{\natexlab{b}})Carmon, Duchi, Hinder, and
  Sidford]{carmon2017lowerII}
Yair Carmon, John~C Duchi, Oliver Hinder, and Aaron Sidford.
\newblock Lower bounds for finding stationary points {II}: First-order methods.
\newblock \emph{arXiv preprint arXiv:1711.00841}, 2017{\natexlab{b}}.

\bibitem[Diakonikolas and Guzm{\'a}n(2018)]{diakonikolas2018lower}
Jelena Diakonikolas and Crist{\'o}bal Guzm{\'a}n.
\newblock Lower bounds for parallel and randomized convex optimization.
\newblock \emph{arXiv preprint arXiv:1811.01903}, 2018.

\bibitem[Drori(2017)]{drori2017exact}
Yoel Drori.
\newblock The exact information-based complexity of smooth convex minimization.
\newblock \emph{Journal of Complexity}, 39:\penalty0 1--16, 2017.

\bibitem[Guzm{\'a}n and Nemirovski(2015)]{guzman2015lower}
Crist{\'o}bal Guzm{\'a}n and Arkadi Nemirovski.
\newblock On lower complexity bounds for large-scale smooth convex
  optimization.
\newblock \emph{Journal of Complexity}, 31\penalty0 (1):\penalty0 1--14, 2015.

\bibitem[Juditsky and Nesterov(2014)]{juditsky2014deterministic}
Anatoli Juditsky and Yuri Nesterov.
\newblock Deterministic and stochastic primal-dual subgradient algorithms for
  uniformly convex minimization.
\newblock \emph{Stochastic Systems}, 4\penalty0 (1):\penalty0 44--80, 2014.

\bibitem[Nemirovski and Yudin(1983)]{nemirovski1983problem}
A.~Nemirovski and D.~Yudin.
\newblock \emph{Problem complexity and method efficiency in optimization}.
\newblock Wiley-Interscience Series in Discrete Mathematics. John Wiley, XV,
  1983.

\bibitem[Nemirovski(1992)]{nemirovski1992information}
A.~S. Nemirovski.
\newblock Information-based complexity of linear operator equations.
\newblock \emph{Journal of Complexity}, 8\penalty0 (2):\penalty0 153--175,
  1992.

\bibitem[Nesterov(2004)]{nesterov2004introductory}
Y.~E. Nesterov.
\newblock \emph{Introductory Lectures on Convex Optimization: A Basic Course}.
\newblock Kluwer Academic Publishers, Massachusetts, 2004.

\bibitem[Ouyang and Xu(2019)]{ouyang2018lower}
Yuyuan Ouyang and Yangyang Xu.
\newblock Lower complexity bounds of first-order methods for convex-concave
  bilinear saddle-point problems.
\newblock \emph{Mathematical Programming}, pages 1--35, 2019.

\bibitem[Woodworth and Srebro(2016)]{woodworth2016tight}
Blake~E Woodworth and Nati Srebro.
\newblock Tight complexity bounds for optimizing composite objectives.
\newblock In \emph{Advances in neural information processing systems}, pages
  3639--3647, 2016.

\end{thebibliography}

\end{document}